\documentclass[a4paper]{article}

\usepackage[english]{babel}
\usepackage[utf8x]{inputenc}
\usepackage[T1]{fontenc}
\usepackage[a4paper,top=3cm,bottom=2cm,left=3cm,right=3cm,marginparwidth=1.75cm]{geometry}

\usepackage{amsthm}
\usepackage{amssymb}
\usepackage{amsmath}
\usepackage{mathtools}
\usepackage{bbm}
\usepackage{tikz-cd}

\newtheorem{theorem}{Theorem}[section]
\newtheorem{lemma}[theorem]{Lemma}
\newtheorem{definition}[theorem]{Definition}
\newtheorem{proposition}[theorem]{Proposition}

\newtheorem{remark}[theorem]{Remark}

\newcommand{\N}{\mathbb{N}}
\newcommand{\T}{\mathbb{T}}

\newcommand{\Z}{\mathbb{Z}}
\newcommand{\R}{\mathbb{R}}

\newcommand{\PS}{\mathcal{PS}}

\DeclarePairedDelimiter\floor{\lfloor}{\rfloor}

\numberwithin{equation}{section}

\title{\textbf{Fast ergodicity of rotations on the circle and stability of one-dimensional non-periodic Sturmian ground states}}

\author{Damian G\l odkowski \\  Institute of Mathematics \\ Polish Academy of Sciences \\ \'Sniadeckich 8,  00-656 Warsaw, Poland\\
d.glodkowski@uw.edu.pl
\\ \\ Jacek Mi\c{e}kisz \\ Institute of Applied Mathematics and Mechanics \\ University of Warsaw \\ Banacha 2, 02-097 Warsaw, Poland \\ miekisz@mimuw.edu.pl}

\begin{document}

\maketitle

\begin{abstract}
Rotations on the circle by irrational numbers give rise to uniquely ergodic Sturm dynamical systems. 
We show that rotations by badly approximable irrationals have the property of fast ergodicity. 
It was shown recently that any Sturmian ergodic measure is the unique ground state of a non-frustrated Hamiltonian of one-dimensional classical lattice-gas model.
We use the fast ergodicity property to show that for slowly decaying interactions, $1/r^{\alpha}$ with $1 < \alpha < 3/2$, non-periodic Sturmian ground states are 
stable with respect to periodic configurations consisting of Sturmian words.
\end{abstract}

\section{Introduction}\label{introduction}
It was recently shown that Sturmian sequences generated by irrational rotations may be uniquely determined by the absence of finite patterns including pairs of 1's 
at so-called forbidden distances and the zero-word of length $m$ for some $m\in \N$ \cite{ahj}. This construction gives rise to one-dimensional, non-frustrated, 
lattice-gas (essentially) two-body Hamiltonians with unique ground-state measures supported by Sturmian sequences. 
We consider the problem of the stability of those ground-states with respect to finite-range perturbations of interactions.

First we address the problem of the speed of convergence in the ergodic theorem. One of the classic papers here is \cite{speed1}, 
recently there appeared two papers on arXiv \cite{speed2,speed3}. Our result here concerns fast ergodicity behavior of rotations on the circle 
in the case of badly approximable irrationals, Theorem \ref{badly-theorem}. We use it to prove partial stability of Sturmian ground states 
generated by rotation by badly approximable irrationals if interactions between pairs of 1's decay slowly with distance, $1/r^{\alpha}$ with $1 < \alpha < 3/2$,
Theorem \ref{main}.

We hope that this result can be extended to the full stability since we compare the energy in Sturmian sequences to the periodic ones with no forbidden patterns inside the period (so they seem to have the lowest energy among periodic sequences). However, the general problem remains open. 

\section{Sturmian systems}\label{sequences}

We will consider bi-infinite sequences (words) of two symbols $\{0,1\}$, i.e. elements of $\Omega = \{0,1\}^\Z$.
We will identify the circle $C$ with $\R/\Z$ and consider an irrational rotation by $\varphi$ (which is given by translation on $\R/\Z$ by $\varphi \mod 1$).   

\begin{definition}
Given an irrational $\varphi\in C$ we say that $X\in\{0,1\}^\Z$ is \textbf{generated by $\varphi$} if it is of the following form: 
\begin{equation*}
    X(n) = \begin{cases}
               0               & \text{when} \ x+n\varphi \in P \\
               1                & \text{otherwise}
           \end{cases}
\end{equation*}
where $x\in C$ and $P=[0,\varphi)$ or $P=(0,\varphi]$.
\end{definition}

We call such $X$ a Sturmian sequence corresponding to $\varphi$. Let $T$ be the translation operator, i.e., $T:\Omega \rightarrow \Omega, (T(X))(i) = X(i-1), X \in \Omega$. 
Let $G_{St}$ be the closure (in the product topology of the discrete topology on $\{0,1\}$) of the orbit of $X$ by translations, 
i.e., $G_{St} = \{T^{n} (X), n \geq 0\}^{cl}.$ It can be shown that $G_{St}$ supports exactly one translation-invariant probability measure 
$\rho_{St}$ on $\Omega$. 

We call $(\Omega, T, \rho_{St})$ a uniquely ergodic symbolic dynamical system.

A frequency of a finite pattern in an infinite configuration is defined as the limit of the number
of occurrences of this pattern in a segment of length $L$ divided by $L$ as $L \rightarrow \infty$. All sequences
in any given Sturmian system have the same frequency of any given pattern. We say that configurations satisfy 
the {\bf strict boundary condition} \cite{sbc1} 
or rapid convergence of frequencies to their equilibrium values
\cite{sbc2} if fluctuations of the numbers of occurrences 
of patterns on segments are bounded. It is known 
that Sturmian sequences satisfy 
the strict boundary condition (cf. \cite[Theorem 3.3]{ahj}).

We need one more characterization of Sturmian words in terms of patterns that don't appear in given sequence.

\begin{theorem}\label{forbidden}\emph{\cite[Theorem 4.1]{ahj}}
Let $\varphi\in (\frac{1}{2},1)$ be irrational. Then there exist a natural number $m$ and a set $F\subseteq \N$ of forbidden distances such that Sturmian words generated by $\varphi$ are uniquely determined by the absence of the following patterns: $m$ consecutive 0's and two 1's separated by a distance from $F$.
\end{theorem}

To characterize Sturmian words generated by irrationals from $(0,\frac{1}{2})$ we have to change the roles of 0's and 1's. We will show that $F$ can also be described by the rotation. 

\begin{proposition}\label{forbidden-pro}
The set $F$ from Theorem \ref{forbidden} may be chosen in the following way: 
$$k\notin F \iff \exists y\in [\varphi , 1) \ y+k\varphi \in [\varphi, 1),$$
or equivalently 
$$ k\in F \iff k\varphi \in [1-\varphi, \varphi].  $$
\end{proposition}
\begin{proof}
We start with the proof of equivalence of the above statements. Let's see that 
$$\neg (\exists y\in [\varphi , 1) \ y+z \in [\varphi, 1)) \iff \forall y\in[\varphi, 1) \ y+z\in [0,\varphi) \iff \{y+z: y\in [\varphi,1)\} \subseteq [0,\varphi). $$
The arc $\{y+z: y\in [\varphi,1)\}$ is contained in $[0, \varphi)$ if and only if its endpoints are in $[0,\varphi]$ which means that $z \in [1-\varphi, \varphi]$. Picking $z=k\varphi$ completes the proof of this part. 

For the proof that $F$ is a good set of forbidden distances proceed as follows. Fix a Sturmian word $X$ generated by $\varphi$ with an initial point $x$. If $X(n)=X(n+k)=1$ then $x+n\varphi\in [\varphi,1)$ and $ x+(n+k)\varphi \in [\varphi,1)$, so for $y=x+n\varphi$ we have $y, y+k\varphi \in [\varphi,1)$ which proves that $k\notin F$. Conversely, assume that $k\notin F$. Let $y$ be such that $y, y+k\varphi \in [\varphi,1)$. Then $Y(0)=Y(k)$ where $Y$ is the sequence generated by $\varphi$ with initial point $y$, so $k$ is not a forbidden distance. 
\end{proof}

Given a set of forbidden distances we may easily construct non-frustrated Hamiltonians for which the unique ground-state consists exactly of Sturmian words generated by $\varphi$. Simply we need to assign positive energies to all forbidden patterns and zero otherwise (for more details see \cite[Theorem 5.2]{ahj}). 

\section{Badly approximable numbers and fast ergodicity}\label{badly-app}

In most of the results we will consider rotations by badly approximable numbers, which have better ergodic behavior.  

\begin{definition}\label{badly-def}
We say that a number $\varphi$ is \textbf{badly approximable} if there exists $c>0$ such that 
$$\left | \varphi - \frac{p}{q} \right | > \frac{c}{q^2} $$
for all rationals $\frac{p}{q}$. 
\end{definition}

From the classical theorem of Liouville it follows that algebraic numbers of degree 2 are badly approximable. 

\begin{theorem}
(Liouville) If $\varphi$ is an algebraic number of degree $d>1$ then there exists $c>0$ such that 
$$\left | \varphi - \frac{p}{q} \right | > \frac{c}{q^d} $$
for all rationals $\frac{p}{q}$.
\end{theorem}

The above theorem was originally used to show that there exist transcendental numbers by finding numbers which do not satisfy the above inequality for any $d$. In particular there are numbers which are not badly approximable. However, the class of badly approximable numbers is much larger than the class of algebraic numbers of degree 2. Both classes can be naturally described in the terms of continued fractions (cf. \cite[Proposition 3.10, Theorem 3.13]{Einsiedler}).

Now we reformulate the definition of badly approximable numbers to apply them in the context of rotations. 

\begin{lemma}\label{badly-lemma}
Assume that $\varphi\in \T$ is badly approximable. Then there exists $c>0$ such that for each $k\in \N$ we have $k\varphi \in (\frac{c}{k}, 1-\frac{c}{k})$. 
\end{lemma}

\begin{proof}
Multiplication of both sides of the inequality from the Definition \ref{badly-def} by $q$ gives $$|q\varphi -p|>\frac{c}{q}. $$
Now by putting $q=k$ and $p=\floor{k\varphi}$ we get 
$$|k\varphi \mod 1| = |k\varphi-\floor{k\varphi}| > \frac{c}{k} $$
which completes the proof. 
\end{proof}

The next theorem will play a crucial role in the proof of partial stability of Sturmian ground state. We will see that ergodicity of rotation by a badly approximable number can be reached quickly i.e. for given arc $P$ of length greater than $\frac{1}{2}$ we are in $P$ frequently enough after a small number of rotations by $k\varphi$ for $k\in\N$. 

\begin{theorem}\label{badly-theorem}
Let $\varphi\in \T$ be badly approximable and $P\subseteq \T$ be the arc of length at least $\frac{1}{2}$. For given $x_0\in \T$ and each $k\in \N$ define sequences $(x_i^k)_{i\in\N}$ in the following way: 
\begin{itemize}
      \item $x_0^k=x_0$,
    \item $x_{i+1}^k=x_i^k+k\varphi$. 
\end{itemize}
Then there exist $r>0$ and $d\in \N$ (independent of $x_0$) such that $$ |\{x_i^k: i\in \{1,2,\dots, dk\}, x_i^k\in P\}| \geq rk$$ for sufficiently large $k$. 
\end{theorem}
\begin{proof}
Fix a natural number $k$. To simplify the notation let us put $x_i=x_i^k$. Without loss of generality we can assume that $P$ contains $[\frac{1}{2},1)$ and that $k\varphi\in (0,\frac{1}{2})$ since the other case is similar. 

Let $c>0$ be as in Lemma \ref{badly-lemma} and $n\in \N$ be such that $k\varphi \in (\frac{1}{n+1}, \frac{1}{n})$. We can assume that $c=\frac{1}{d}$ for some $d\in \N$ so we have $$k\varphi\in \left (\frac{1}{dk}, 1-\frac{1}{dk}\right ).$$
In particular $2\leq n\leq dk$. We consider three cases: 
\begin{enumerate}
    \item $4\leq n\leq \frac{dk}{2}-1$. \\
    We observe that if $x_i\notin [\frac{1}{2},1)$ and $x_{i+1}\in [\frac{1}{2},1)$ then 
    $$x_{i+1}\in \left [\frac{1}{2},\frac{1}{2}+\frac{1}{n} \right ), x_{i+2}\in \left [\frac{1}{2},\frac{1}{2}+\frac{2}{n} \right ), \dots, x_{i+s}\in \left [\frac{1}{2},\frac{1}{2}+\frac{s}{n} \right )$$
    where $s=\floor{\frac{n}{2}}$, so we also have $x_{i+j}\in [\frac{1}{2},1)\subseteq P$ for $j=1,2,\dots, s$. If we do $n+1$ rotations by $k\varphi$ we will get one full rotation jointly so we will hit the interval $P$ at least $s$ times. Hence after $dk$ rotations we hit $P$ at least $\floor{\frac{dk}{n+1}}s$ times which gives an estimate 
    $$|\{x_i: i\in \{1,2,\dots, dk\}, x_i\in P\}| \geq \left \lfloor \frac{dk}{n+1} \right \rfloor s \geq \left (\frac{dk}{n+1}-1 \right ) \left ( \frac{n}{2}-1 \right ). $$
    Consider the function $$f(x)=(\frac{dk}{x+1}-1)(\frac{x}{2}-1).$$ Then we have 
    \begin{gather*}
        \frac{1}{dk}f^{\prime\prime}(x)= -\frac{3}{(x+1)^3}<0
    \end{gather*}
    for $x>0$ so $f$ is concave on $(0,\infty)$. Hence we can estimate $f(x)\geq \min \{f(4), f(\frac{dk}{2}-1)\}$ when $x\in [4,\frac{dk}{2}-1]$. We see that 
    $$f(4)=(\frac{dk}{5}-1)(2-1)=\frac{dk-5}{5}\geq \frac{d}{6}k$$
    and 
    $$f(\frac{dk}{2}-1)=(2-1)(\frac{dk-6}{4})=\frac{dk-6}{4}\geq \frac{d}{6}k$$
    for sufficiently large $k$. In particular $$|\{x_i: i\in \{1,2,\dots, dk\}, x_i\in P\}| \geq f(n)\geq \frac{d}{6}k. $$
    \item $\frac{dk}{2}-1< n\leq dk$.\\
    Again, if we rotate $dk$ times by $k\varphi$ we will hit $P$ at least $s=\floor{\frac{n}{2}}$ times so we get $$|\{x_i: i\in \{1,2,\dots, dk\}, x_i\in P\}| \geq \left \lfloor \frac{n}{2} \right \rfloor >\frac{n}{2}-1>\frac{dk-2}{4}-1\geq \frac{d}{6}k$$ for sufficiently large $k$.
    \item $2\leq n\leq 4$. \\
    We observe that for every $i$ at least one of the consecutive points $x_i, x_{i+1}, x_{i+2}, x_{i+3}$ is in $P$. Indeed, $k\varphi>\frac{1}{5}$ so if $x_i, x_{i+1}, x_{i+2}\in [0,\frac{1}{2})$, then $x_{i+1}\in (\frac{1}{5}, \frac{1}{2})$, $x_{i+2}\in (\frac{2}{5}, \frac{1}{2})$ and so $x_{i+3}\in (\frac{3}{5},1)\subseteq P$. Hence we get estimate 
    $$|\{x_i: i\in \{1,2,\dots, dk\}, x_i\in P\}| \geq \left \lfloor \frac{d}{4}k \right \rfloor\geq \frac{d}{6}k.$$
\end{enumerate}
To complete the proof it is enough to put $r=\frac{d}{6}$.
\end{proof}

\section{Stability of Sturmian ground states}

As we mentioned at the end of Section \ref{sequences} we may view Sturmian words as zero-energy configurations of non-frustrated Hamiltonians. We will consider a natural family of such Hamiltonians, namely for $\alpha>1$ let $H_\alpha$ denote the Hamiltonian which assigns energy $\frac{1}{n^\alpha}$ to each pair of 1's in forbidden distance $n$ and energy 1 to the forbidden sequence of consecutive 0's. We would like to know when those Hamiltonians are stable against small perturbations in finite range. We start with the necessary definitions. 

\begin{definition}\label{energy-density}
For a one-dimensional Hamiltonian $H$ and a finite word $w$ denote by $H(w)$ joint energy of all patterns in $w$ given by $H$. For $X\in \{0,1\}^\Z$ we define the \textbf{energy density} of $X$ as the limit 
$$\rho_H(X)=\liminf_{k \rightarrow \infty} \frac{ H(X([-k,k]))}{2k+1}$$
where $X([-k,k])=(X(i))_{i=-k}^k$.
\end{definition}

In the light of Theorem \ref{forbidden} energy density of Sturmian words is zero (which is minimal possible) when they are considered as ground-state configurations of $H_\alpha$. 
It is well-known that ground-state configurations of any Hamiltonian minimize the energy density (cf. \cite[Theorem 5]{stable}). 

\begin{definition}
Let $H$ be a Hamiltonian and $P$ be a finite set of patterns. For $\lambda>0$ we say that a Hamiltonian $\widetilde{H}$ is a $(\lambda,P)$\textbf{-perturbation} of $H$ if for each pattern $p\in P$ we have $|\Phi(p)-\widetilde{\Phi}(p)|<\lambda$ and $\Phi(q)=\widetilde{\Phi}(q)$ when $q\notin P$, where $\Phi(p), \widetilde{\Phi}(p)$ denote the energy of $p$ given by $H$ and $\widetilde{H}$ respectively.
\end{definition}

\begin{definition}
Let $H$ be a one-dimensional Hamiltonian and $\mathcal{W}\subseteq \{0,1\}^\Z$ be some set of bi-infinite words. We say that a ground-state configuration $X\in\{0,1\}^\Z$ of $H$ is \textbf{stable with respect to} $\mathcal{W}$ (against small perturbations) if for every finite set of patterns $P$ there exists $\lambda>0$ such that if $\widetilde{H}$ is a $(\lambda,P)$-perturbation of $H$, then $\rho_{\widetilde{H}}(X)\leq \rho_{\widetilde{H}}(W)$ for all $W\in\mathcal{W}$. 

If $X$ is stable with respect to the whole space $\{0,1\}^\Z$ then we say that $X$ is \textbf{stable} (against small perturbations). 
\end{definition}

Our aim is to show that if $\alpha$ is small enough then Sturmian words considered as ground-state configurations of $H_\alpha$ are stable with respect to periodic sequences which are close to Sturmian ones. Those sequences have a small number of forbidden patterns inside so we expect that their energy density is lower than the energy density of other periodic sequences. 

\begin{definition}
We say that a sequence $X\in\{0,1\}^\Z$ is \textbf{periodically Sturmian} word generated by $\varphi$ if it is periodic and its period is a subword of a Sturmian sequence generated by $\varphi$. We denote by $\PS_\varphi$ the set of all periodically Sturmian words generated by $\varphi$.
\end{definition}

As we showed in \cite{nonstability} Sturmian words are not stable with respect to $\PS_\varphi$ when $\alpha > 3$, so it is natural to ask what happens 
when $\alpha$ is smaller.

We can now formulate the main theorem of our work. 

\begin{theorem}\label{main}
Assume that $\varphi\in (\frac{3}{4},1)$ is badly approximable and that $\alpha\leq \frac{3}{2}$. If $X$ is a Sturmian word generated by $\varphi$ then $X$ 
is stable with respect to $\PS_\varphi$ for the Hamiltonian $H_\alpha$. 
\end{theorem}

For the proof of this theorem we will use a lemma which will allow us to simplify counting energy density in the case of periodic sequences. We will call a Hamiltonian $H$ summable if there exists a constant $M>0$ such that for every $i\in \Z$ we have $$\sum_{p\in P_i} |\Phi(p)|<M,$$ where $\Phi(p)$ is the energy of $p$ given by $H$ and $P_i$ is the set of all finite patterns containing symbol at the $i$-th coordinate. Since $\sum_k \frac{1}{k^\alpha} <\infty$ for $\alpha>1$ Hamiltonians $H_\alpha$ and all their $(\lambda,P)$-perturbations are summable, so the lemma applies. 

\begin{lemma}\label{lemma-density}
Let $H$ be a one-dimensional summable Hamiltonian.  Then for $k\in \N$ and $X\in\{0,1\}^\Z$ we have
\begin{gather*}
    \rho_H(X)=\liminf_{m\rightarrow \infty} \frac{H(X([-mk,mk]))}{2mk+1}.
\end{gather*}\end{lemma}
\begin{proof}
Fix $k\in \N$ and $X\in\{0,1\}^\Z$. For each $n\in \N$ let $a_n, b_n$ be such that $n=a_nk+b_n$ and $0\leq b_n< k$.
For $n\in \N$ define $$F_n=\sum_{p\in Q_n} |\Phi(p)|$$ where $\Phi(p)$ is again the energy of pattern $p$ and $Q_n$ is the set of finite patterns which contain the symbol at the zero position but are not contained in the segment $[-n,n]$ (since $\Phi$ is translation invariant we can also take $i$-th position and the segment $[-n+i,n+i]$ for any $i\in\Z$ instead). By summability $F_n\rightarrow 0$ when $n\rightarrow \infty$. We see also that $$|H(X([-s,s])-H(X([-t,t])|\leq 2\sum_{i=0}^sF_i$$ whenever $s\geq t$. 
Hence for each $n$ we get 
$$\left |\frac{H(X[-n,n])}{2n+1} - \frac{H(X[-a_nk,a_nk])}{2n+1}\right |\leq \frac{2\sum_{i=0}^n F_i}{2n+1}\xrightarrow{n\rightarrow \infty} 0$$
since the limit of arithmetic means of $F_i$ is equal to the limit of $F_i$. This gives $$\liminf_{n\rightarrow \infty} \frac{H(X([-a_nk,a_nk]))}{2n+1} = \rho_H(X). $$
Moreover $$\lim_{n\rightarrow \infty} \frac{2a_nk+1}{2n+1}=1 $$
so $$\liminf_{n\rightarrow \infty} \frac{H(X([-a_nk,a_nk]))}{2a_nk+1} = \rho_H(X). $$
\end{proof}

\begin{remark}\label{remark-density}
If the frequency of 1's in a sequence $Y\in \{0,1\}^\Z$ is sufficiently small then there are many forbidden patterns in $Y$ consisting of consecutive 0's so the energy density of $Y$ is high even after small perturbation of the Hamiltonian, so we can omit this case. From now on we will assume that there exists $l>0$ such that for each considered $Y\in \PS_\varphi$ we have $\xi_1(Y)\geq l$, where $\xi_1(Y)$ is the frequency of 1's in $Y$.
\end{remark}

Now we can move on to the proof of Theorem \ref{main}.

\begin{proof}
Let $P=\{p_1, p_2, \dots, p_n\}$ be a fixed set of finite patterns. For simplicity denote $H=H_\alpha$ and let $\widetilde{H}=H+H_P$ be a $(\lambda,P)$-perturbation of $H$ for some $\lambda>0$. Let $X$ be a Sturmian word generated by $\varphi$ and for each $k\in \N$ let $Y_k\in \PS_\varphi$ be a sequence of period $k$. We may assume that $Y_k([1,k])$ is a subword of a Sturmian word. In the light of Lemma \ref{lemma-density} we need to show that for small enough $\lambda>0$ and each $k\in\N$ the energy $\widetilde{H}(Y_k([-mk,mk]))$ is not less then the energy $\widetilde{H}(X([-mk,mk]))$ for $m$ large enough. 

For a pattern $p$ denote by $|p|$ its length. Since Sturmian sequences satisfy the strict boundary condition (cf. \cite[Theorem 3.3]{ahj}), there are constants $D_i$ such that for every $s\in \Z$
\begin{gather*}
    |n_{p_i}(X([sk,(s+1)k]))-n_{p_i}(Y_k([sk,(s+1)k]))|<D_i.
\end{gather*}
Extending the segment $[sk,(s+1)k]$ by $|p_i|$ on both sides we get
\begin{gather*}
    |n_{p_i}(X([sk-|p_i|,(s+1)k+|p_i|]))-n_{p_i}(Y_k([sk-|p_i|,(s+1)k+|p_i|]))|<D_i+2|p_i|.
\end{gather*}
Since positions of each pattern $p_i$ are contained in the segment of the form $[sk-|p_i|, (s+1)k+|p_i|]$ we have
\begin{gather*}
    |n_{p_i}(X([-mk,mk]))-n_{p_i}(Y_k([-mk,mk]))|<2m(D_i+2|p_i|)\leq mC
\end{gather*}
where $C=2\max_i (D_i+2|p_i|)$. Hence we get 
\begin{gather}\label{Hp-energy}
    |H_P(X([-mk,mk]))-H_P(Y_k([-mk,mk]))|\leq \sum_{i=1}^n 2m(D_i+2|p_i|)\lambda\leq nmC\lambda.
\end{gather}
Now we need to estimate $H(Y_k([-mk,mk])$ from below by finding many pairs of 1's in $Y_k$. 
Fix $y,t\in [1,k]$ such that $Y_k(y)=Y_k(y+t)=1$. It follows from the periodicity that  $Y_k(y+t+sk)=1$ for each $s\in\Z$. By Proposition \ref{forbidden-pro} a pair of 1's in positions $y$ and $y+t+sk$ is forbidden if and only if $(t+sk)\varphi\in [1-\varphi, \varphi]$. Define the sequence $x_s^k$ by $x_0^k=t\varphi$ and $x_s^k=x_{s-1}^k+k\varphi=(t+sk)\varphi$. By assumption $\varphi\in (3/4,1)$ so the arc $[1-\varphi,\varphi]$ has length greater than $1/2$ and we can apply Theorem \ref{badly-theorem}. Let $d\in \N, r>0$ be such that $$ |\{x_i^k: i\in \{1,2,\dots, dk\}, x_i^k\in P\}| \geq rk$$ for sufficiently large $k$. Then there are at least $rk$ forbidden pairs of 1's at a distance less than $dk+t$ and the energy of those pairs can be estimated from below by $$\frac{rk}{(t+dk^2)^\alpha}\sim \frac{c}{k^{2\alpha -1}} $$
where $c=r/d^\alpha$. For segment $A\subseteq \Z$ let $E(A)$ denote the energy of forbidden pairs of 1's such that position of one of those 1's is in $[1,k]$. Numbers $y$ and $t$ can be chosen in at least $lk$ ways (cf. Remark \ref{remark-density}) and each forbidden pair is counted at most 2 times so
\begin{gather*}
    E([-(dk^2+t), dk^2+t+k])\geq \frac{rk}{(t+dk^2)^\alpha}\frac{(lk)^2}{2} \sim \frac{c_1}{k^{2\alpha-3}}= c_1k^{3-2\alpha}
\end{gather*}
for $c_1=\frac{1}{2}cl^2$. Since $[-(dk^2+t), dk^2+t+k]\subset [-(d+2)k^2,(d+2)k^2]$ we get that 
\begin{gather*}
    E([-(d+2)k^2,(d+2)k^2])\gtrsim c_1k^{3-2\alpha}.
\end{gather*}
The same argument gives the same estimate for energy of pairs of 1's in $[sk-(d+2)k^2,sk+(d+2)k^2]$ such that position of one of those 1's is in $[sk+1,(s+1)k]$ for $s\in \Z$. For $|s|\leq m-(d+2)$ we have  $[sk-(d+2)k^2,sk+(d+2)k^2]\subset [-mk,mk]$ so taking all the mentioned pairs of 1's we get the estimate
\begin{gather*}
    H(Y_k([-mk,mk])) \gtrsim (m-(d+2)k) c_1k^{3-2\alpha}.
\end{gather*}
If $m$ is large enough we have 
\begin{gather}\label{H-energy}
    H(Y_k([-mk,mk])) \gtrsim \frac{1}{2} m c_1k^{3-2\alpha}.
\end{gather}
Since $H(X([-mk,mk]))=0$, combining (\ref{Hp-energy}) and (\ref{H-energy}) gives 
\begin{gather*}
H(Y_k([-mk,mk]))-H(X([-mk,mk])) \gtrsim \frac{1}{2} m c_1k^{3-2\alpha} - nmC\lambda. 
\end{gather*}
By assumption $\alpha\leq 3/2$ so $\frac{1}{2} c_1k^{3-2\alpha}$ is bounded from below by some positive constant. Hence if $\lambda$ is very small then the right side of above inequality is positive which completes the proof. 
\end{proof}

Another natural class of words that could possibly indicate instability of Sturmian ground states generated by $\varphi$ is the class of Sturmian words generated by $\psi$, where $\psi$ is close to $\varphi$.

Denote by $\mathcal{S}_\psi$ the class of Sturmian words generated by $\psi$ and put $$\mathcal{S}= \bigcup_{n\in\N} \mathcal{S}_{\varphi-\frac{1}{n}} .$$ Then we have the following theorem.

\begin{theorem}
Assume that $\alpha\leq 2$. If $X$ is a Sturmian word generated by $\varphi$ then $X$ is stable with respect to $\mathcal{S}$ for the Hamiltonian $H_\alpha$. 
\end{theorem}

\begin{proof}
For simplicity we will consider only small perturbations of $H_\alpha$ of the form $\widetilde{H}_\alpha= H_\alpha + H_1$ (where $H_1$ is a Hamiltonian favoring the presence of 1's). The proof of the general case is similar. 
 
Let $S_\varphi$ be a Sturmian word generated by $\varphi$.
For $n\in \N, n>0$ consider the Sturmian sequence $S_n\in \mathcal{S}_{\varphi-\frac{1}{n}}$ given by
\begin{equation*}
    S_n(k) = \begin{cases}
               0               & \text{when} \ k\varphi \in [0,\varphi-\frac{1}{n}) \\
               1                & \text{otherwise}
           \end{cases}
\end{equation*}

    We will show that for any $\lambda=H_1(1)$ and big enough $n$, the energy density of $S_n$ with respect to $\widetilde{H}_\alpha$ is bigger than the energy density of $S_\varphi$.


    For $a\in \Z$ such that $S_n(a)=1$, let $E(a)$ denote the energy of pairs of 1's that include 1 at position $a$. 
Without loss of generality we may assume that $a=0$ and put $E=E(0)$. Then by Proposition \ref{forbidden-pro} a symbol 1 at position $k$ forms a forbidden pair with 1 at position 0 if and only if 
$$k\varphi \mod 1 \in [1-\varphi, \varphi].$$

    Hence elements at positions $0$ and $k$ in $S_n$ form a forbidden pair if and only if the following are satisfied simultaneously: 
    \begin{gather}\label{*}
     \begin{cases} 
    k(\varphi-1/n) \mod 1 \in [\varphi-1/n,1) \\
    k\varphi \mod 1 \in [1-\varphi, \varphi]
    \end{cases}      
    \end{gather}
Fix $\varepsilon$ such that $0<\varepsilon<1-\varphi$ (it is important that $\varepsilon$ depends only on $\varphi$). Then if $k$ satisfies the following 
 \begin{gather}
     \begin{cases} 
    k/n \mod 1 \in [\varepsilon,1-\varphi) \\
    k\varphi \mod 1 \in [\varphi-\varepsilon, \varphi]
    \end{cases}
    \end{gather}
    then $k$ also satisfies (\ref{*}). Among $\{1/n, 2/n, 3/n, \dots, n/n\}$ there are $\lfloor n(1-\varphi-\varepsilon) \rfloor$ consecutive numbers belonging to the interval $[\varepsilon, 1-\varphi)$. Hence and by the ergodicity of the rotation by $\varphi$ we have  
$$|\{k: k\in [1,n], k/n\in [\varepsilon, 1-\varphi), k\varphi\in [\varphi-\varepsilon,\varphi)\}| > cn$$
    for large enough $n$ and some $c$ independent of $n$. This means that every 1 in $S_n$ belongs to at least $cn$ forbidden pairs with distances at most $n$. Thus, the energy contribution from these pairs may be estimated from below by 
    $$\frac{cn}{n^\alpha}.$$
    Since the density of 1's in $S_n$ is equal to $1-\varphi+\frac{1}{n}$, the average energy contribution from forbidden pairs is may be estimated from below by 
    $$\frac{c_1n}{n^\alpha},$$
    where $c_1=c(1-\varphi)$, while the average energy from the additional 1's in the sequence equals $\lambda/n$. 
    Since $\alpha\leq 2$, if $\lambda$ is small enough, then for large enough $n$ we have 
$$\frac{c_1n}{n^\alpha}> \frac{\lambda}{n},$$
    so the energy density of $S_\varphi$ is smaller than the energy density of $S_n$ for large enough $n$.
\end{proof}

{\bf Acknowledgments} We would like to thank the National Science Centre (Poland) for a financial support under Grant No. 2016/22/M/ST1/00536.

\end{document}